\documentclass[12pt]{amsart}

\usepackage{mathpazo}

\theoremstyle{plain}
\newtheorem{theorem}{Theorem}[section]
\newtheorem{proposition}[theorem]{Proposition}
\newtheorem{lemma}[theorem]{Lemma}

\newtheorem{claim}[theorem]{Claim}
\newtheorem{fact}[theorem]{Fact}

\theoremstyle{definition}
\newtheorem{definition}[theorem]{Definition}

\theoremstyle{remark}

\newcommand{\uv}[1]{\textquotedblleft #1\textquotedblright}

\begin{document}

\title{The isomorphism class of $ c_{0} $ is not Borel}
\author{Ond\v{r}ej Kurka}
\thanks{The author was supported by grant GA\v{C}R 17-00941S and by RVO: 67985840.}
\address{Mathematical Institute of Czech Academy of Sciences, \v{Z}itn\'{a} 25, 115 67 Prague 1, Czech Republic}
\email{kurka.ondrej@seznam.cz}
\begin{abstract}
We show that the class of all Banach spaces which are isomorphic to $ c_{0} $ is a complete analytic set with respect to the Effros Borel structure of separable Banach spaces. The proof employs a recent Bourgain-Delbaen construction by Argyros, Gasparis and Motakis.
\end{abstract}
\maketitle

\section{Introduction and the main result}

In a preceding work \cite{kurka}, we have introduced a new approach to complexity problems in Banach space theory based on the famous Tsirelson space. Using a method from \cite{arggasmot}, we construct a family of $ \mathcal{L}_{\infty} $-spaces that is somehow related to the family of Tsirelson-like spaces from \cite{argdel}. Compared to \cite{kurka}, this provides an analogous but in some sense more powerful method and enables us to determine the complexity of several classes of separable Banach spaces.

It can be shown quite easily that the isomorphism class of any separable Banach space is analytic with respect to the Effros Borel structure. B.~Bossard asked in \cite{bossard} whether $ \ell_{2} $ is (up to isomorphism) the only infinite-dimensional separable Banach space whose isomorphism class is Borel (see Section~\ref{sec:prelim} for the definitions of the Effros Borel structure and of the related notions used below).

Although Bossard's question has been answered negatively (see \cite[Theorem~6.2]{godefroycompl}), it is still not well understood which spaces have a Borel isomorphism class. It is known that the isomorphism class is not Borel for Pe\l czy\'nski's universal space (see \cite[Theorem 2.3]{bossard}), $ C(2^{\mathbb{N}}) $ (see e.g. \cite[(33.26)]{kechris}) or $ L_{p}([0, 1]), 1 < p < \infty, p \neq 2, $ (see \cite[p.~130]{bossard} or \cite[Corollary~4.10]{ghawadrah}), and further examples are provided in \cite{godefroycompl}. On the other hand, the isomorphism classes of $ \ell_{p}, 1 < p < \infty, $ are Borel (see \cite[Theorem~2]{godefroycompl2}, see also \cite{gosr}). It is not known if the same holds for $ \ell_{1} $.

G.~Godefroy asked in \cite{godefroyprobl} if the isomorphism class of the space $ c_{0} $ is Borel. In the preceding paper \cite{kurka}, we have found a partial solution, proving that the class of all spaces isomorphic to a subspace of $ c_{0} $ is not Borel. The main result of the present work is a full solution of Godefroy's problem.

\begin{theorem} \label{thmmain}
The class of all Banach spaces isomorphic to $ c_{0} $ is complete analytic. In particular, it is not Borel.
\end{theorem}

A remarkable conjecture from \cite{gokala} states that if a Banach space $ X $ has summable Szlenk index and its dual $ X^{*} $ is isomorphic to $ \ell_{1} $, then $ X $ is isomorphic to $ c_{0} $. In \cite{godefroyprobl}, G.~Godefroy pointed out that validity of the conjecture would imply that the isomorphism class of $ c_{0} $ is Borel. Therefore, it is not surprising that our proof of Theorem~\ref{thmmain} is based on a counterexample to the conjecture that has been found recently by S.~A.~Argyros, I.~Gasparis and P.~Motakis \cite{arggasmot}.

As noticed in \cite{godefroyprobl}, we obtain immediately that the class of isomorphic subspaces of $ c_{0} $ is not Borel (that is the mentioned result from \cite{kurka}), as the intersection of this class with the Borel class of all infinite-dimensional separable $ \mathcal{L}_{\infty} $-spaces is exactly the isomorphism class of $ c_{0} $ (see \cite[Corollary~1]{johzip}). This works in one direction only. Theorem~\ref{thmmain} is an improvement indeed, since it is basically a statement about the structure of the class of $ \mathcal{L}_{\infty} $-spaces.

Actually, it follows from our method of proving Theorem~\ref{thmmain} that the isomorphism class of $ c_{0} \oplus X $ is not Borel in the case that $ X $ is not big in a sense (in fact, we do not know if there is a space $ X $ such that the isomorphism class of $ c_{0} \oplus X $ is Borel).

\begin{theorem} \label{thmmain2}
Let $ X $ be a separable Banach space that satisfies any of the following conditions:
\begin{itemize}
\item $ X $ does not contain an isomorphic copy of $ c_{0} $,
\item $ X $ does not contain an infinite-dimensional reflexive subspace,
\item $ X $ is a subspace of a space with an unconditional basis.
\end{itemize}
Then the class of all Banach spaces isomorphic to $ c_{0} \oplus X $ is not Borel. Furthermore, the class of all Banach spaces isomorphic to a subspace of $ c_{0} \oplus X $ is not Borel as well.
\end{theorem}

Let us remark that the second part of this result has been already established for $ X $ without infinite-dimensional reflexive subspaces in \cite[Remark~3.9(i)]{kurka}.

The following theorem contains further consequences of Theorem~\ref{thmmain} pointed out by G.~Godefroy in \cite{godefroyprobl}.

\begin{theorem} \label{thmmain3}
The following classes of separable Banach spaces are analytic but not Borel:
\begin{itemize}
\item spaces with an unconditional basis,
\item subspaces of spaces with an unconditional basis,
\item spaces isomorphic to a $ C(K) $ space.
\end{itemize}
\end{theorem}

Let us note that it is not known whether the class of all spaces with a Schauder basis is Borel. On the other hand, the class of all separable Banach spaces with the bounded approximation property is known to be Borel (see \cite{ghawadrah2}).

Let us mention some achievements of Banach space theory employed in this work. One of our main tools is the space $ T $ introduced by T.~Figiel and W.~B.~Johnson (see \cite{figjoh}), dual to the famous (original) Tsirelson space $ T^{*} $ (see \cite{tsirelson}). The construction of $ T $ was later generalized by S.~A.~Argyros and I.~Deliyanni (see \cite{argdel}). Their Tsirelson-like spaces have been used in \cite{kurka} for proving that the class of isomorphic subspaces of $ c_{0} $ is not Borel. These spaces play a role also in the present paper, although their application is not so direct this time.

The major part of the paper is devoted to the construction of a $ \mathcal{L}_{\infty} $-space $ \mathfrak{X}_{\mathcal{M}} $ for a compact system $ \mathcal{M} $ of sets of natural numbers. The gist of this construction is that $ \mathfrak{X}_{\mathcal{M}} $ is isomorphic to $ c_{0} $ if and only if $ \mathcal{M} $ contains an infinite set. Theorem~\ref{thmmain} follows quite easily, as the set of all compact systems $ \mathcal{M} $ containing an infinite set is not Borel by a classical result of W.~Hurewicz.

The space $ \mathfrak{X}_{\mathcal{M}} $ is a modification of the space $ \mathfrak{X}_{0} $ from \cite{arggasmot}. Both spaces are constructed by the Bourgain-Delbaen method introduced in the seminal paper \cite{bourdelb}. In fact, the constructions of $ \mathfrak{X}_{\mathcal{M}} $ and $ \mathfrak{X}_{0} $ are very similar, which enables us to keep most of the used notation. The main difference is that the growth condition $ n \geq (\#\Gamma_{\mathrm{rank}(\xi)})^{2} $ from \cite[p.~696]{arggasmot} is replaced with a condition involving $ \mathcal{M} $.

\section{Preliminaries} \label{sec:prelim}

Our terminology concerning Banach space theory and descriptive set theory follows \cite{fhhmpz} and \cite{kechris}. All Banach spaces in this paper are considered over $ \mathbb{R} $. By an isomorphism we always mean a linear isomorphism.

Given $ \lambda \geq 1 $, we say that Banach spaces $ F $ and $ G $ are \emph{$ \lambda $-isomorphic} if there is a surjective linear operator $ T : F \to G $ such that $ \Vert T \Vert \Vert T^{-1} \Vert \leq \lambda $.

A Banach space $ X $ is called a \emph{$ \mathcal{L}_{\infty, \lambda} $-space} if, for any finite-dimensional subspace $ F $ of $ X $, there exists a finite-dimensional subspace $ G $ of $ X $ containing $ F $ such that $ G $ is $ \lambda $-isomorphic to $ \ell_{\infty}^{n} $, where $ n = \mathrm{dim} \, G $. A Banach space $ X $ is said to be a \emph{$ \mathcal{L}_{\infty} $-space} if it is a $ \mathcal{L}_{\infty, \lambda} $-space for some $ \lambda \geq 1 $.

By $ c_{00} $ we denote the vector space of all systems $ x = \{ x(n) \}_{n=1}^{\infty} $ of scalars such that $ x(n) = 0 $ for all but finitely many $ n $'s. By the canonical basis of $ c_{00} $ we mean the algebraic basis consisting of vectors $ e_{n} = \mathbf{1}_{\{ n \}}, n \in \mathbb{N} $. For $ E \subset \mathbb{N} $ and $ x \in c_{00} $, we denote by $ Ex $ the element of $  c_{00} $ given by $ Ex(n) = x(n) $ for $ n \in E $ and $ Ex(n) = 0 $ for $ n \notin E $.

In the context of Banach spaces, by a basis we mean a Schauder basis. A basis $ \{ x_{i} \}_{i=1}^{\infty} $ of a Banach space $ X $ is said to be \emph{unconditional} if there is a constant $ c \geq 1 $ such that $ \Vert \sum_{i \in A} a_{i}x_{i} \Vert \leq c \Vert \sum_{i \in B} a_{i}x_{i} \Vert $ whenever $ A \subset B $ are finite sets of natural numbers and $ a_{i} \in \mathbb{R} $ for $ i \in B $. A basis $ \{ x_{i} \}_{i=1}^{\infty} $ of a Banach space $ X $ is said to be \emph{shrinking} if
$$ X^{*} = \overline{\mathrm{span}} \{ x_{1}^{*}, x_{2}^{*}, \dots \} $$
where $ x_{1}^{*}, x_{2}^{*}, \dots $ is the dual basic sequence $ x_{n}^{*} : \sum_{i=1}^{\infty} a_{i}x_{i} \mapsto a_{n} $.

We will need the following standard fact (see e.g. \cite[Lemma~B.6]{dodos}), as well as some immediate consequences.

\begin{lemma} \label{lemmrams}
Let $ X_{1} $, $ X_{2} $ be Banach spaces and let $ P_{X_{i}}, i = 1, 2, $ denote the projection $ (x_{1}, x_{2}) \in X_{1} \oplus X_{2} \mapsto x_{i} $. If $ Y $ is an infinite-dimensional subspace of $ X_{1} \oplus X_{2} $, then there are $ i \in \{ 1, 2 \} $ and an infinite-dimensional subspace $ Z $ of $ Y $ such that $ P_{X_{i}}|_{Z} $ is an isomorphism.
\end{lemma}

\begin{fact} \label{factrams}
{\rm (1)} If both $ X_{1} $, $ X_{2} $ do not contain an isomorphic copy of $ c_{0} $, then the same holds for $ X_{1} \oplus X_{2} $.

{\rm (2)} If both $ X_{1} $, $ X_{2} $ do not contain an infinite-dimensional reflexive subspace, then the same holds for $ X_{1} \oplus X_{2} $.
\end{fact}

We will need also the following result from \cite{rosenthal} uncovering the isomorphic structure of the space $ c_{0} $.

\begin{theorem}[Rosenthal] \label{thmros}
Let $ X $ be a $ \mathcal{L}_{\infty} $-space. If $ X $ is isomorphic to a subspace of a space with an unconditional basis, then $ X $ is isomorphic to $ c_{0} $.
\end{theorem}

A \emph{Polish space (topology)} means a separable completely metrizable space (topology). A set $ P $ equipped with a $ \sigma $-algebra is called a \emph{standard Borel space} if the $ \sigma $-algebra is generated by a Polish topology on $ P $.

A subset $ A $ of a standard Borel space $ X $ is called \emph{analytic} if there are a standard Borel space $ Z $ and a Borel mapping $ g : Z \to X $ such that $ A = g(Z) $. Moreover, a subset $ A $ of a standard Borel space $ X $ is called a \emph{hard analytic set} if every analytic subset $ B $ of a standard Borel space $ Y $ admits a Borel mapping $ f : Y \to X $ such that $ f^{-1}(A) = B $. A subset of a standard Borel space is called a \emph{complete analytic set} if it is analytic and hard analytic at the same time.

Let us recall a standard simple argument for a set to be hard analytic.

\begin{lemma} \label{lemmhardset}
Let $ A \subset X $ and $ C \subset Z $ be subsets of standard Borel spaces $ X $ and $ Z $. Assume that $ C $ is hard analytic. If there is a Borel mapping $ g : Z \to X $ such that
$$ g(z) \in A \quad \Leftrightarrow \quad z \in C, $$
then $ A $ is hard analytic as well.
\end{lemma}

By $ \mathcal{P}(\mathbb{N}) $ we denote the set of all subsets of $ \mathbb{N} $ endowed with the coarsest topology for which $ \{ A \in \mathcal{P}(\mathbb{N}) : n \in A \} $ is clopen for every $ n $. Obviously, $ \mathcal{P}(\mathbb{N}) $ is nothing else than a copy of the Cantor space $ \{ 0, 1 \}^{\mathbb{N}} $.

For a topological space $ X $, we denote by $ F(X) $ the family of all closed subsets and by $ K(X) $ the family of all compact subsets of $ X $.

The \emph{hyperspace of compact subsets of $ X $} is defined as $ K(X) $ equipped with the \emph{Vietoris topology}, i.e., the topology generated by the sets of the form 
$$ \{ K \in K(X) : K \subset U \}, $$
$$ \{ K \in K(X) : K \cap U \neq \emptyset \}, $$
where $ U $ varies over open subsets of $ X $. If $ X $ is Polish, then so is $ K(X) $.

The set $ F(X) $ can be equipped with the \emph{Effros Borel structure}, defined as the $ \sigma $-algebra generated by the sets
$$ \{ F \in F(X) : F \cap U \neq \emptyset \}, $$
where $ U $ varies over open subsets of $ X $. If $ X $ is Polish, then, equipped with this $ \sigma $-algebra, $ F(X) $ forms a standard Borel space.

It is well-known that the space $ C([0, 1]) $ contains an isometric copy of every separable Banach space. By the \emph{standard Borel space of separable Banach spaces} we mean
$$ \mathcal{SE}(C([0, 1])) = \big\{ F \in F(C([0, 1])) : \textrm{$ F $ is linear} \big\}, $$
considered as a subspace of $ F(C([0, 1])) $.

By \cite[Proposition~2.2]{bossard}, $ \mathcal{SE}(C([0, 1])) $ is a standard Borel space. Whenever we say that a class of separable Banach spaces has a property like being Borel, analytic, complete analytic, etc., we consider that class as a subset of $ \mathcal{SE}(C([0, 1])) $.

The following lemma is needed for proving that the parametrized construction in Section~\ref{sec:bourdelb} can be realizable as a Borel mapping.

\begin{lemma}[{\cite[Lemma 2.4]{kurka}}] \label{lemmselect}
Let $ \Xi $ be a standard Borel space and let $ \{ (X_{\xi}, \Vert \cdot \Vert_{\xi}) \}_{\xi \in \Xi} $ be a system of Banach spaces each member of which contains a sequence $ x^{\xi}_{1}, x^{\xi}_{2}, \dots $ whose linear span is dense in $ X_{\xi} $. Assume that the function
$$ \xi \in \Xi \; \mapsto \; \Big\Vert \sum_{k=1}^{n} \lambda_{k} x^{\xi}_{k} \Big\Vert_{\xi} \in \mathbb{R} $$
is Borel whenever $ n \in \mathbb{N} $ and $ \lambda_{1}, \dots, \lambda_{n} \in \mathbb{R} $. Then there exists a Borel mapping $ \mathfrak{S} : \Xi \to \mathcal{SE}(C([0, 1])) $ such that $ \mathfrak{S}(\xi) $ is isometric to $ X_{\xi} $ for every $ \xi \in \Xi $.
\end{lemma}

\section{Tsirelson-like spaces}

In this section, we recall the definition of Tsirelson-like spaces introduced by S. A. Argyros and I. Deliyanni \cite{argdel}. The structure of these spaces enables us to prove some properties of a compact system of finite sets of natural numbers.

For $ \mathcal{M} \in K(\mathcal{P}(\mathbb{N})) $, a family $ \{ E_{1}, \dots, E_{n} \} $ of successive finite subsets of $ \mathbb{N} $ is said to be \emph{$ \mathcal{M} $-admissible} if an element of $ \mathcal{M} $ contains numbers $ m_{1}, \dots, m_{n} $ such that
$$ m_{1} \leq E_{1} < m_{2} \leq E_{2} < \dots < m_{n} \leq E_{n}. $$

\begin{definition}[Argyros, Deliyanni]
For $ \mathcal{M} \in K(\mathcal{P}(\mathbb{N})) $, the space $ T[\mathcal{M}, \frac{1}{2}] $ is defined as the completion of $ c_{00} $ under the implicitly defined norm
$$ \Vert x \Vert_{\mathcal{M}, \frac{1}{2}} = \max \bigg\{ \Vert x \Vert_{\infty}, \frac{1}{2} \sup \sum_{k=1}^{n} \Vert E_{k}x \Vert_{\mathcal{M}, \frac{1}{2}} \bigg\}, $$
where the \uv{sup} is taken over all $ \mathcal{M} $-admissible families $ \{ E_{1}, \dots, E_{n} \} $.
\end{definition}

\begin{lemma} \label{lemmaTsA}
Let $ \mathcal{M} \in K(\mathcal{P}(\mathbb{N})) $ consist of finite sets only. Let $ \omega > 0 $. Then there are $ k \in \mathbb{N} $ and $ \alpha_{1}, \dots, \alpha_{k} \geq 0 $ with $ \sum_{i=1}^{k} \alpha_{i} = 1 $ such that
$$ \forall A \in \mathcal{M} : \sum_{1 \leq i \leq k, i \in A} \alpha_{i} < \omega. $$
\end{lemma}

\begin{proof}
First of all, since $ \mathcal{M} $ consists of finite sets, the canonical basis $ e_{n} = \mathbf{1}_{\{ n \}} $ of $ c_{00} $ is a shrinking basis of $ T[\mathcal{M}, \frac{1}{2}] $ (to show this, it is possible to adapt the part (a) of the proof of \cite[Proposition~1.1]{argdel} if we consider $ \theta_{k} = \frac{1}{2} $ and $ \mathcal{M}_{k} = \mathcal{M} $ for every $ k $).

In particular, $ e_{1}, e_{2}, \dots $ is not equivalent to the standard basis of $ \ell_{1} $. We can find $ k \in \mathbb{N} $ and real numbers $ \beta_{1}, \dots, \beta_{k} $ such that
$$ \Big\Vert \sum_{i=1}^{k} \beta_{i} e_{i} \Big\Vert_{\mathcal{M}, \frac{1}{2}} < \frac{1}{2} \omega \cdot \sum_{i=1}^{k} |\beta_{i}|. $$
Considering $ \alpha_{j} = |\beta_{j}|/\sum_{i=1}^{k} |\beta_{i}| $, we obtain $ \alpha_{i} \geq 0 $, $ \sum_{i=1}^{k} \alpha_{i} = 1 $ and
$$ \Big\Vert \sum_{i=1}^{k} \alpha_{i} e_{i} \Big\Vert_{\mathcal{M}, \frac{1}{2}} = \frac{1}{\sum_{i=1}^{k} |\beta_{i}|} \Big\Vert \sum_{i=1}^{k} |\beta_{i}| e_{i} \Big\Vert_{\mathcal{M}, \frac{1}{2}} = \frac{1}{\sum_{i=1}^{k} |\beta_{i}|} \Big\Vert \sum_{i=1}^{k} \beta_{i} e_{i} \Big\Vert_{\mathcal{M}, \frac{1}{2}} < \frac{1}{2} \omega. $$
Given $ A \in \mathcal{M} $, let $ m_{1} < m_{2} < \dots < m_{n} $ be all elements of $ A $ with $ m_{j} \leq k $ and let $ E_{j} = \{ m_{j} \} $. Then the family $ \{ E_{1}, \dots, E_{n} \} $ is $ \mathcal{M} $-admissible and we can write
$$ \frac{1}{2} \omega > \Big\Vert \sum_{i=1}^{k} \alpha_{i} e_{i} \Big\Vert_{\mathcal{M}, \frac{1}{2}} \geq \frac{1}{2} \sum_{j=1}^{n} \Big\Vert E_{j}\sum_{i=1}^{k} \alpha_{i} e_{i} \Big\Vert_{\mathcal{M}, \frac{1}{2}} = \frac{1}{2} \sum_{j=1}^{n} \alpha_{m_{j}} = \frac{1}{2} \sum_{1 \leq i \leq k, i \in A} \alpha_{i}. $$
Therefore, the numbers $ \alpha_{1}, \dots, \alpha_{k} $ work.
\end{proof}

\begin{lemma} \label{lemmaTsB}
Let $ \mathcal{M} \in K(\mathcal{P}(\mathbb{N})) $ consist of finite sets only. Let $ \omega > 0 $. Let $ E_{1}, E_{2}, \dots $ be a sequence of intervals of $ \mathbb{N} \cup \{ 0 \} $ such that
$$ E_{1} < E_{2} < \dots . $$
Then there are $ k \in \mathbb{N} $ and $ \alpha_{1}, \dots, \alpha_{k} \geq 0 $ with $ \sum_{i=1}^{k} \alpha_{i} = 1 $ such that
$$ \forall A \in \mathcal{M} : \sum_{1 \leq i \leq k, E_{i} \cap A \neq \emptyset} \alpha_{i} < \omega. $$
\end{lemma}

\begin{proof}
Let us consider a mapping
$$ A \in \mathcal{P}(\mathbb{N}) \mapsto A' = \{ i : E_{i} \cap A \neq \emptyset \} \in \mathcal{P}(\mathbb{N}). $$
Then $ \mathcal{M}' = \{ A' : A \in \mathcal{M} \} $, being a continuous image of a compact set, is compact itself. It consists of finite sets only, and so Lemma~\ref{lemmaTsA} can be applied. There are $ k \in \mathbb{N} $ and $ \alpha_{1}, \dots, \alpha_{k} \geq 0 $ with $ \sum_{i=1}^{k} \alpha_{i} = 1 $ such that
$$ \forall B \in \mathcal{M}' : \sum_{1 \leq i \leq k, i \in B} \alpha_{i} < \omega. $$
Clearly, $ k $ and $ \alpha_{1}, \dots, \alpha_{k} $ work.
\end{proof}

\section{A Bourgain-Delbaen construction} \label{sec:bourdelb}

We are going to introduce a construction of a $ \mathcal{L}_{\infty} $-space which is the key ingredient of our proof of Theorem~\ref{thmmain}. This space is based on a recent example from \cite{arggasmot}, with the novelty that a compact system of sets of natural numbers is involved as a parameter.

As well as the authors did in \cite[Sect.~5]{arggasmot}, we fix a natural number $ N \geq 3 $ and a constant $ 1 < \theta < N/2 $. We moreover fix a constant $ 0 < \omega < 1/\theta $. Given $ \mathcal{M} \in K(\mathcal{P}(\mathbb{N})) $, we consider the following objects:
\begin{itemize}
\item $ \Delta_{0} = \{ 0 \} $,
\item a sequence $ \Delta_{1}, \Delta_{2}, \dots $ of pairwise disjoint non-empty finite sets that is defined recursively below,
\item to every $ \gamma \in \bigcup_{i=1}^{\infty} \Delta_{i} $, we assign $ \mathrm{age}(\gamma) \in \{ 1, \dots, N \} $ in a way described below,
\item for every $ \gamma \in \bigcup_{i=0}^{\infty} \Delta_{i} $, we denote by $ \mathrm{rank}(\gamma) $ the unique $ p $ so that $ \gamma \in \Delta_{p} $,
\item we put $ \Gamma_{p} = \bigcup_{i=0}^{p} \Delta_{i} $ for $ p = 0, 1, \dots $,
\item a functional $ c_{\gamma}^{*} \in (\ell_{\infty}(\Gamma_{p-1}))^{*} $ is defined below for every $ p \geq 1 $ and $ \gamma \in \Delta_{p} $,
\item an extension operator $ i_{p-1, p} : \ell_{\infty}(\Gamma_{p-1}) \to \ell_{\infty}(\Gamma_{p}) $, where $ p \geq 1 $, is defined by
$$ i_{p-1, p}(x)(\gamma) = \left\{\begin{array}{ll}
x(\gamma), & \quad \gamma \in \Gamma_{p-1}, \\
c_{\gamma}^{*}(x), & \quad \gamma \in \Delta_{p},
\end{array} \right. $$
\item for $ 0 \leq p < q $, we define $ i_{p, q} = i_{q-1, q} \circ \dots \circ i_{p, p+1} : \ell_{\infty}(\Gamma_{p}) \to \ell_{\infty}(\Gamma_{q}) $ and $ i_{p, p} $ as the identity on $ \ell_{\infty}(\Gamma_{p}) $,
\item for $ 0 \leq p \leq q $, let $ r_{q, p} : \ell_{\infty}(\Gamma_{q}) \to \ell_{\infty}(\Gamma_{p}) $ be the corresponding restriction operator,
\item for $ 0 \leq p \leq s \leq q $, let us consider the projections on $ \ell_{\infty}(\Gamma_{q}) $ given by $ P_{[0, p]}^{q} = i_{p, q} \circ r_{q, p} $ and $ P_{(p, s]}^{q} = P_{[0, s]}^{q} - P_{[0, p]}^{q} $.
\end{itemize}

For $ q = 0, 1, \dots $, we define $ \Delta_{q+1} $ as the set of all tuples of one of the following two forms:
$$ (q+1, \{ \varepsilon_{i} \}_{i=1}^{k}, \{ E_{i} \}_{i=1}^{k}, \{ \eta_{i} \}_{i=1}^{k}, \{ h_{i} \}_{i=1}^{k}), \leqno \textrm{(a)} $$
where $ k \in \mathbb{N} $, $ \varepsilon_{i} \in \{ -1, 1 \} $ for $ i = 1, \dots, k $, $ \{ E_{i} \}_{i=1}^{k} $ is a sequence of successive non-empty intervals of $ \{ 0, \dots, q \} $, $ \eta_{i} \in \Gamma_{q} $ and $ h_{i} \in \mathbb{N} \cup \{ 0 \} $ for $ i = 1, \dots, k $, and these numbers satisfy the requirements
$$ \sum_{i=1}^{k} h_{i} \leq q+1 $$
and
$$ \forall A \in \mathcal{M} : \sum_{E_{i} \cap A \neq \emptyset} h_{i} < \omega \cdot (q+1), $$
or, if $ q \geq 2 $,
$$ (q+1, \xi, \{ \varepsilon_{i} \}_{i=1}^{k}, \{ E_{i} \}_{i=1}^{k}, \{ \eta_{i} \}_{i=1}^{k}, \{ h_{i} \}_{i=1}^{k}), \leqno \textrm{(b)} $$
where $ k \in \mathbb{N} $, $ \xi \in \Gamma_{q-1} \setminus \Gamma_{0} $ with $ \mathrm{age}(\xi) < N $, $ \varepsilon_{i} \in \{ -1, 1 \} $ for $ i = 1, \dots, k $, $ \{ E_{i} \}_{i=1}^{k} $ is a sequence of successive non-empty intervals of $ \{ \mathrm{rank}(\xi) + 1, \dots, q \} $, $ \eta_{i} \in \Gamma_{q} $ and $ h_{i} \in \mathbb{N} \cup \{ 0 \} $ for $ i = 1, \dots, k $, and these numbers satisfy the same two requirements as above.

Moreover, we define $ \mathrm{age}(\gamma) $ and $ c_{\gamma}^{*} \in (\ell_{\infty}(\Gamma_{q}))^{*} $ for $ \gamma \in \Delta_{q+1} $ as follows (by $ e^{*}_{\eta} $ we mean the functional $ x \in \ell_{\infty}(\Gamma_{q}) \mapsto x(\eta) $).

For $ \gamma \in \Delta_{q+1} $ of the form (a), we set $ \mathrm{age}(\gamma) = 1 $ and define
$$ c_{\gamma}^{*} = \frac{\theta}{N} \frac{1}{q+1} \sum_{i=1}^{k} h_{i} \varepsilon_{i} e^{*}_{\eta_{i}} \circ P_{E_{i}}^{q}. $$
For $ \gamma \in \Delta_{q+1} $ of the form (b), we set $ \mathrm{age}(\gamma) = \mathrm{age}(\xi) + 1 $ and define
$$ c_{\gamma}^{*} = e^{*}_{\xi} + \frac{\theta}{N} \frac{1}{q+1} \sum_{i=1}^{k} h_{i} \varepsilon_{i} e^{*}_{\eta_{i}} \circ P_{E_{i}}^{q}. $$

Finally,
\begin{itemize}
\item let $ \Gamma = \bigcup_{i=0}^{\infty} \Delta_{i} $,
\item for $ p \geq 0 $, let $ i_{p} : \ell_{\infty}(\Gamma_{p}) \to \ell_{\infty}(\Gamma) $ be given by
$$ i_{p}(x)(\gamma) = \left\{\begin{array}{ll}
x(\gamma), & \quad \gamma \in \Gamma_{p}, \\
i_{p, q}(x)(\gamma), & \quad \gamma \in \Delta_{q}, q > p,
\end{array} \right. $$
(let us point out that $ i_{p} $ maps into $ \ell_{\infty}(\Gamma) $ by Claim~\ref{XMclaim1} below),
\item let $ d_{\gamma} = i_{\mathrm{rank}(\gamma)}(e_{\gamma}) \in \ell_{\infty}(\Gamma) $ for $ \gamma \in \Gamma $, where $ e_{\gamma} $ denotes the basic vector $ \mathbf{1}_{\{ \gamma \}} $ of $ \ell_{\infty}(\Gamma_{\mathrm{rank}(\gamma)}) $,
\item we define $ \mathfrak{X}_{\mathcal{M}} = \overline{\mathrm{span}} \, \{ d_{\gamma} : \gamma \in \Gamma \} $,
\item for $ p \geq 0 $, let $ r_{p} : \mathfrak{X}_{\mathcal{M}} \to \ell_{\infty}(\Gamma_{p}) $ be the corresponding restriction operator,
\item for $ 0 \leq p \leq s $, let us consider the projections on $ \mathfrak{X}_{\mathcal{M}} $ given by $ P_{[0, p]} = i_{p} \circ r_{p} $, $ P_{(p, s]} = P_{[0, s]} - P_{[0, p]} $, $ P_{(p, \infty)} = I - P_{[0, p]} $ and $ P_{[0, \infty)} = I $,
\item we denote by $ \{ d^{*}_{\gamma} \}_{\gamma \in \Gamma} $ the system in $ \mathfrak{X}_{\mathcal{M}}^{*} $ orthogonal to $ \{ d_{\gamma} \}_{\gamma \in \Gamma} $, that is $ d^{*}_{\gamma}(x) = P_{\{ \mathrm{rank}(\gamma) \}}(x)(\gamma) $ for $ x \in \mathfrak{X}_{\mathcal{M}} $,
\item for $ x \in \mathfrak{X}_{\mathcal{M}} $, let $ \mathrm{range} \, x $ be the smallest interval $ E $ of $ \mathbb{N} \cup \{ 0 \} $ such that $ P_{E}(x) = x $.
\end{itemize}

We provide a series of claims which results in a characterization of those $ \mathcal{M} $'s for which $ \mathfrak{X}_{\mathcal{M}} $ is isomorphic to $ c_{0} $ (Proposition~\ref{propXMisomc0}). The first of them is more or less standard, as it is based on an argument from \cite{bourdelb} (see also e.g. \cite[Theorem~3.4]{arghay}). Nevertheless, we include a proof for the convenience of a reader which is not familiar with constructions of this kind.

Let us note that it follows from the claim that $ i_{p} $ witnesses that $ \ell_{\infty}(\Gamma_{p}) $ and $ \mathrm{span} \, \{ d_{\gamma} : \gamma \in \Gamma_{p} \} $ are $ [N/(N-2\theta)] $-isomorphic. For this reason, $ \mathfrak{X}_{\mathcal{M}} $ is a $ \mathcal{L}_{\infty} $-space.

\begin{claim} \label{XMclaim1}
We have $ \Vert i_{p, s} \Vert \leq N/(N-2\theta) $, and consequently $ \Vert i_{p} \Vert \leq N/(N-2\theta) $, $ \Vert P_{(p, s]}^{q} \Vert \leq 2N/(N-2\theta) $, $ \Vert P_{(p, s]} \Vert \leq 2N/(N-2\theta) $ and $ \Vert d^{*}_{\gamma} \Vert \leq 2N/(N-2\theta) $ whenever $ 0 \leq p \leq s \leq q $ and $ \gamma \in \Gamma $.
\end{claim}

\begin{proof}
We show by induction on $ s $ that $ \Vert i_{p, s} \Vert \leq N/(N-2\theta) $ for every $ p \leq s $. Assume that $ p \leq s $ and that $ \Vert i_{t, q} \Vert \leq N/(N-2\theta) $ whenever $ t \leq q \leq s-1 $. We need to show that $ |i_{p, s}(x)(\gamma)| \leq \Vert x \Vert \cdot N/(N-2\theta) $ for $ x \in \ell_{\infty}(\Gamma_{p}) $ and $ \gamma \in \Gamma_{s} $. If $ \gamma \in \Gamma_{p} $, then $ i_{p, s}(x)(\gamma) = x(\gamma) $. So, let $ \gamma \in \Gamma_{s} \setminus \Gamma_{p} $. There is $ q $ with $ p \leq q \leq s-1 $ such that $ \gamma \in \Delta_{q+1} $. Note that
$$ i_{p, s}(x)(\gamma) = i_{p, q+1}(x)(\gamma) = c^{*}_{\gamma}(i_{p, q}(x)). $$
We need to check the following two simple facts first:\\
(i) $ \Vert P_{E}^{q}(i_{p, q}(x)) \Vert \leq \Vert x \Vert \cdot 2N/(N-2\theta) $ for an interval $ E $ of $ \{ 0, 1, \dots, q \} $,\\
(ii) $ P_{E}^{q}(i_{p, q}(x)) = 0 $ for an interval $ E $ of $ \{ p+1, \dots, q \} $.

If $ p \leq t \leq q $, then $ P_{[0, t]}^{q}(i_{p, q}(x)) = i_{t, q}(r_{q, t}(i_{p, q}(x))) = i_{t, q}(i_{p, t}(x)) = i_{p, q}(x) $. From this, (ii) follows. Moreover, by the induction hypothesis, $ \Vert P_{[0, t]}^{q}(i_{p, q}(x)) \Vert = \Vert i_{p, q}(x) \Vert \leq \Vert i_{p, q} \Vert \Vert x \Vert \leq \Vert x \Vert \cdot N/(N-2\theta) $. To check (i), it remains to show that $ \Vert P_{[0, t]}^{q}(i_{p, q}(x)) \Vert \leq \Vert x \Vert \cdot N/(N-2\theta) $ also for $ t < p $.

If $ 0 \leq t < p $, then $ P_{[0, t]}^{q}(i_{p, q}(x)) = i_{t, q}(r_{q, t}(i_{p, q}(x))) = i_{t, q}(r_{p, t}(x)) $, and so $ \Vert P_{[0, t]}^{q}(i_{p, q}(x)) \Vert \leq \Vert i_{t, q} \Vert \Vert r_{p, t}(x) \Vert \leq \Vert x \Vert \cdot N/(N-2\theta) $.

So, (i) and (ii) are checked. Now, if $ \gamma $ is of the form (a), then
\begin{align*}
|c^{*}_{\gamma}(i_{p, q}(x))| & \leq \frac{\theta}{N} \frac{1}{q+1} \sum_{i=1}^{k} h_{i} \cdot \Vert x \Vert \cdot \frac{2N}{N-2\theta} \\
 & \leq \Vert x \Vert \cdot \frac{2\theta}{N-2\theta} \leq \Vert x \Vert \cdot \frac{N}{N-2\theta}.
\end{align*}
If $ \gamma $ is of the form (b) and $ \mathrm{rank}(\xi) \leq p $, then $ e^{*}_{\xi}(i_{p, q}(x)) = x(\xi) $ and
\begin{align*}
|c^{*}_{\gamma}(i_{p, q}(x))| & \leq \Vert x \Vert + \frac{\theta}{N} \frac{1}{q+1} \sum_{i=1}^{k} h_{i} \cdot \Vert x \Vert \cdot \frac{2N}{N-2\theta} \\
 & \leq \Vert x \Vert + \Vert x \Vert \cdot \frac{2\theta}{N-2\theta} = \Vert x \Vert \cdot \frac{N}{N-2\theta}.
\end{align*}
If $ \gamma $ is of the form (b) and $ \mathrm{rank}(\xi) > p $, then
$$ |c^{*}_{\gamma}(i_{p, q}(x))| = |e^{*}_{\xi}(i_{p, q}(x))| \leq \Vert x \Vert \cdot \Vert i_{p, q} \Vert \leq \Vert x \Vert \cdot \frac{N}{N-2\theta}. $$
In all three cases, $ |i_{p, s}(x)(\gamma)| = |c^{*}_{\gamma}(i_{p, q}(x))| \leq \Vert x \Vert \cdot N/(N-2\theta) $.
\end{proof}

\begin{claim}[{cf. \cite[Proposition~5.3]{arggasmot}}] \label{XMclaim2}
Let $ q \geq 0 $ and $ \gamma \in \Delta_{q+1} $. If $ \gamma $ is of the form {\rm (a)}, then
$$ e_{\gamma}^{*} = d_{\gamma}^{*} + \frac{\theta}{N} \frac{1}{q+1} \sum_{i=1}^{k} h_{i} \varepsilon_{i} e^{*}_{\eta_{i}} \circ P_{E_{i}}. $$
If $ \gamma $ is of the form {\rm (b)}, then
$$ e_{\gamma}^{*} = e^{*}_{\xi} + d_{\gamma}^{*} + \frac{\theta}{N} \frac{1}{q+1} \sum_{i=1}^{k} h_{i} \varepsilon_{i} e^{*}_{\eta_{i}} \circ P_{E_{i}}. $$
(Here, by $ e^{*}_{\eta} $ we mean the functional $ x \in \mathfrak{X}_{\mathcal{M}} \mapsto x(\eta) $).
\end{claim}

\begin{proof}
It is sufficient to consider only the form (b) for $ \gamma $. To distinguish the functionals $ e^{*}_{\eta} $ acting on $ \Gamma_{q} $ from those acting on $ \mathfrak{X}_{\mathcal{M}} $, we use the notation $ e^{*}_{\eta, \Gamma_{q}} $.

Let us pick $ x \in \mathfrak{X}_{\mathcal{M}} $. Note first that $ e^{*}_{\gamma}(P_{[0, q]}(x)) = e^{*}_{\gamma}(i_{q}(r_{q}(x))) = i_{q}(r_{q}(x))(\gamma) = i_{q, q+1}(r_{q}(x))(\gamma) = c^{*}_{\gamma}(r_{q}(x)) $ and $ e^{*}_{\gamma}(P_{[0, q+1]}(x)) = e^{*}_{\gamma}(i_{q+1}(r_{q+1}(x))) = i_{q+1}(r_{q+1}(x))(\gamma) = x(\gamma) = e^{*}_{\gamma}(x) $. We obtain
$$ e^{*}_{\gamma}(x) = e^{*}_{\gamma}(P_{[0, q+1]}(x)) = e^{*}_{\gamma}(P_{[0, q]}(x)) + e^{*}_{\gamma}(P_{\{ q+1 \}}(x)) = c^{*}_{\gamma}(r_{q}(x)) + d^{*}_{\gamma}(x). $$

Note further that, for $ p \leq q $ and $ \eta \in \Gamma_{q} $, we have
\begin{align*}
(e^{*}_{\eta, \Gamma_{q}} \circ P_{[0, p]}^{q})(r_{q}(x)) & = i_{p, q}(r_{q, p}(r_{q}(x)))(\eta) = i_{p, q}(r_{p}(x))(\eta) \\
 & = i_{p}(r_{p}(x))(\eta) = (e^{*}_{\eta} \circ P_{[0, p]})(x).
\end{align*}
It follows that
\begin{align*}
c_{\gamma}^{*}(r_{q}(x)) & = e^{*}_{\xi, \Gamma_{q}}(r_{q}(x)) + \frac{\theta}{N} \frac{1}{q+1} \sum_{i=1}^{k} h_{i} \varepsilon_{i} (e^{*}_{\eta_{i}, \Gamma_{q}} \circ P_{E_{i}}^{q}) (r_{q}(x)) \\
 & = e^{*}_{\xi}(x) + \frac{\theta}{N} \frac{1}{q+1} \sum_{i=1}^{k} h_{i} \varepsilon_{i} (e^{*}_{\eta_{i}} \circ P_{E_{i}}) (x).
\end{align*}
As this works for any $ x \in \mathfrak{X}_{\mathcal{M}} $, the proof is finished.
\end{proof}

\begin{claim}[{cf. \cite[Proposition~5.6]{arggasmot}}] \label{XMclaim3}
Let $ u_{0}, u_{1}, \dots, u_{m} $ be elements of $ \mathfrak{X}_{\mathcal{M}} $ with $ \Vert u_{j} \Vert \leq 1 $. If there is $ A \in \mathcal{M} $ such that
$$ \mathrm{range} \, u_{0} < n_{1} \leq \mathrm{range} \, u_{1} < n_{2} \leq \dots < n_{m} \leq \mathrm{range} \, u_{m} $$
for some $ n_{1}, \dots, n_{m} \in A $, then
$$ \Big\Vert \sum_{j=0}^{m} u_{j} \Big\Vert \leq C := \frac{1}{1-\theta \omega} \cdot \frac{2N(N+\theta)}{N-2\theta}. $$
\end{claim}

\begin{proof}
Let $ u = \sum_{j=0}^{m} u_{j} $. We prove by induction on $ \mathrm{rank}(\gamma) $ that, for every $ \gamma \in \Gamma $ and every interval $ E $ of $ \mathbb{N} \cup \{ 0 \} $,
$$ | (e^{*}_{\gamma} \circ P_{E}) (u) | \leq C $$
and even
$$ | (e^{*}_{\gamma} \circ P_{E}) (u) | \leq C \cdot \frac{\mathrm{age}(\gamma)}{N} \quad \textrm{if } \mathrm{rank}(\gamma) \geq 1. $$
This is easy when $ \mathrm{rank}(\gamma) = 0 $ (if $ \gamma \in \Delta_{0} $, then $ (e^{*}_{\gamma} \circ P_{E}) (u_{j}) = 0 $ for $ j \geq 1 $). Assume that $ q \in \mathbb{N} \cup \{ 0 \} $ and that the assertion holds for each element of $ \Gamma_{q} $ and each interval of $ \mathbb{N} \cup \{ 0 \} $ and fix $ \gamma \in \Delta_{q+1} $ and an interval $ E $ of $ \mathbb{N} \cup \{ 0 \} $. If $ \gamma $ is of the form (a), then $ \mathrm{age}(\gamma) - 1 = 0 $, and if $ \gamma $ is of the form (b), then $ \mathrm{age}(\gamma) - 1 = \mathrm{age}(\xi) $. Using Claim~\ref{XMclaim2} and the induction hypothesis, we obtain in both cases that $ | (e^{*}_{\gamma} \circ P_{E}) (u) | $ is less than or equal to
$$ C \cdot \frac{\mathrm{age}(\gamma) - 1}{N} + | (d^{*}_{\gamma} \circ P_{E}) (u) | + \frac{\theta}{N} \frac{1}{q+1} \sum_{i=1}^{k} h_{i} | (e^{*}_{\eta_{i}} \circ P_{E_{i} \cap E}) (u) |. $$
Let us show first that
$$ | (d^{*}_{\gamma} \circ P_{E}) (u) | \leq \frac{2N}{N-2\theta}. $$
If $ \mathrm{rank}(\gamma) \notin E $, then $ (d^{*}_{\gamma} \circ P_{E}) (u) = 0 $. Assuming $ \mathrm{rank}(\gamma) \in E $, we have $ (d^{*}_{\gamma} \circ P_{E}) (u) = d^{*}_{\gamma}(u) $. As $ \mathrm{rank}(\gamma) $ belongs to the range of at most one $ u_{j} $, we have $ d^{*}_{\gamma}(u) = d^{*}_{\gamma}(u_{j}) $ for some $ j $. Consequently, $ | (d^{*}_{\gamma} \circ P_{E}) (u) | \leq \Vert d^{*}_{\gamma} \Vert \Vert u_{j} \Vert \leq [2N/(N-2\theta)] \cdot 1 $.

Note that if $ E_{i} \cap A = \emptyset $, then $ E_{i} $ contains no $ n_{j} $, and so it intersects the range of at most one $ u_{j} $. Thus, there is $ j(i) $ such that $ P_{E_{i} \cap E} (u) = P_{E_{i} \cap E}(u_{j(i)}) $. We obtain
\begin{align*}
\sum_{i=1}^{k} h_{i} & | (e^{*}_{\eta_{i}} \circ P_{E_{i} \cap E}) (u) | \\
 & = \sum_{E_{i} \cap A \neq \emptyset} h_{i} | (e^{*}_{\eta_{i}} \circ P_{E_{i} \cap E}) (u) | + \sum_{E_{i} \cap A = \emptyset} h_{i} | (e^{*}_{\eta_{i}} \circ P_{E_{i} \cap E}) (u_{j(i)}) | \\
 & \leq \sum_{E_{i} \cap A \neq \emptyset} h_{i} \cdot C + \sum_{E_{i} \cap A = \emptyset} h_{i} \cdot \Vert P_{E_{i} \cap E} \Vert \Vert u_{j(i)} \Vert \\
 & \leq \omega \cdot (q+1) \cdot C + (q+1) \cdot \frac{2N}{N-2\theta}.
\end{align*}
Therefore,
\begin{align*}
| (e^{*}_{\gamma} \circ P_{E}) (u) | \leq & \, C \cdot \frac{\mathrm{age}(\gamma) - 1}{N} + \frac{2N}{N-2\theta} \\
 & \quad \quad + \frac{\theta}{N} \frac{1}{q+1} \cdot \Big( \omega \cdot (q+1) \cdot C + (q+1) \cdot \frac{2N}{N-2\theta} \Big) \\
 = & \, C \cdot \frac{\mathrm{age}(\gamma) - 1}{N} + \frac{2(N+\theta)}{N-2\theta} + \theta \omega \cdot \frac{C}{N} \\
 = & \, C \cdot \frac{\mathrm{age}(\gamma)}{N}.
\end{align*}
This concludes the proof.
\end{proof}

\begin{claim} \label{XMclaim4}
Suppose that $ \mathcal{M} $ consists of finite sets only. Let $ \varepsilon > 0 $ and let $ x_{1}, x_{2}, \dots $ be a sequence in $ \mathfrak{X}_{\mathcal{M}} $ such that $ \Vert x_{j} \Vert \geq 1 $ and
$$ \mathrm{range} \, x_{1} < m_{1} < \mathrm{range} \, x_{2} < m_{2} < \mathrm{range} \, x_{3} < m_{3} < \dots $$
for some $ m_{1}, m_{2}, \dots $. Then there is $ r \in \mathbb{N} $ such that $ \Vert \sum_{j = 1}^{r} x_{j} \Vert > \theta - \varepsilon $.
\end{claim}

\begin{proof}
We find recursively $ \gamma_{1}, \dots, \gamma_{N} \in \Gamma $ and $ r(1), \dots, r(N) \in \mathbb{N} $ such that $ \mathrm{age}(\gamma_{n}) = n $, $ \mathrm{rank}(\gamma_{n}) = m_{r(n)} $ and $ e^{*}_{\gamma_{n}}(\sum_{j=1}^{r(n)} x_{j}) > \frac{\theta-\varepsilon}{N} \cdot n $. Let $ 0 \leq n \leq N - 1 $ and let there be some $ \gamma_{n} $ and $ r(n) $ in the case that $ n \geq 1 $. We need to find $ \gamma_{n+1} $ and $ r(n+1) $. For practical purposes, we set $ r(0) = 0 $.

For every $ i \in \mathbb{N} $, let us put $ E_{i} = \mathrm{range} \, x_{r(n)+i} $ and let us choose $ \eta_{i} \in \Gamma $ and $ \varepsilon_{i} \in \{ -1, 1 \} $ such that $ \varepsilon_{i} x_{r(n)+i}(\eta_{i}) = |x_{r(n)+i}(\eta_{i})| > 1 - (\varepsilon/2\theta) $. By Lemma~\ref{lemmaTsB}, there are $ k \in \mathbb{N} $ and $ \alpha_{1}, \dots, \alpha_{k} \geq 0 $ with $ \sum_{i=1}^{k} \alpha_{i} = 1 $ such that
$$ \forall A \in \mathcal{M} : \sum_{1 \leq i \leq k, E_{i} \cap A \neq \emptyset} \alpha_{i} < \omega. $$
Let $ q \in \mathbb{N} $ be large enough that $ q+1 > 2k\theta/\varepsilon $, $ E_{i} \subset \{ 0, \dots, q \} $ and $ \eta_{i} \in \Gamma_{q} $ for $ i = 1, \dots, k $, and $ q+1 = m_{r(n+1)} $ for some $ r(n+1) \geq r(n)+k $. Let
$$ h_{i} = \lfloor (q+1)\alpha_{i} \rfloor, \quad i = 1, \dots, k, $$
and
\begin{align*}
\gamma_{n+1} = & \; (q+1, \{ \varepsilon_{i} \}_{i=1}^{k}, \{ E_{i} \}_{i=1}^{k}, \{ \eta_{i} \}_{i=1}^{k}, \{ h_{i} \}_{i=1}^{k}) & \quad \textrm{if } n = 0, \\
\gamma_{n+1} = & \; (q+1, \gamma_{n}, \{ \varepsilon_{i} \}_{i=1}^{k}, \{ E_{i} \}_{i=1}^{k}, \{ \eta_{i} \}_{i=1}^{k}, \{ h_{i} \}_{i=1}^{k}) & \quad \textrm{if } n \geq 1.
\end{align*}
It is easy to check that $ \gamma_{n+1} \in \Delta_{q+1} $. We have $ \mathrm{age}(\gamma_{n+1}) = n+1 $ and $ \mathrm{rank}(\gamma_{n+1}) = q+1 = m_{r(n+1)} $, and it remains to show that $ e^{*}_{\gamma_{n+1}}(\sum_{j=1}^{r(n+1)} x_{j}) > \frac{\theta-\varepsilon}{N} \cdot (n+1) $.

Let us denote $ x = \sum_{j=1}^{r(n+1)} x_{j} $. We realize first that if $ n \geq 1 $, then $ e^{*}_{\gamma_{n}}(x) > \frac{\theta-\varepsilon}{N} \cdot n $. For $ j \geq r(n) + 1 $, the range of $ x_{j} $ is disjoint from $ [0, m_{r(n)}] = [0, \mathrm{rank}(\gamma_{n})] $, and so $ e^{*}_{\gamma_{n}}(x_{j}) = 0 $. Hence, $ e^{*}_{\gamma_{n}}(x) = e^{*}_{\gamma_{n}}(\sum_{j=1}^{r(n)} x_{j}) > \frac{\theta-\varepsilon}{N} \cdot n $.

Using Claim~\ref{XMclaim2}, we obtain in both cases $ n = 0 $ and $ n \geq 1 $ that
\begin{align*}
e_{\gamma_{n+1}}^{*}(x) \geq & \; \frac{\theta-\varepsilon}{N} \cdot n + d_{\gamma_{n+1}}^{*}(x) + \frac{\theta}{N} \frac{1}{q+1} \sum_{i=1}^{k} h_{i} \varepsilon_{i} (e^{*}_{\eta_{i}} \circ P_{E_{i}})(x) \\
 = & \; \frac{\theta-\varepsilon}{N} \cdot n + 0 + \frac{\theta}{N} \frac{1}{q+1} \sum_{i=1}^{k} h_{i} \varepsilon_{i} e^{*}_{\eta_{i}} (x_{r(n)+i}) \\
 > & \; \frac{\theta-\varepsilon}{N} \cdot n + \frac{\theta}{N} \frac{1}{q+1} \sum_{i=1}^{k} h_{i} \cdot \Big( 1 - \frac{\varepsilon}{2\theta} \Big).
\end{align*}
As $ \sum_{i=1}^{k} h_{i} \geq \sum_{i=1}^{k} ((q+1)\alpha_{i} - 1) = q+1 - k > (q+1)(1 - (\varepsilon/2\theta)) $, it follows that
$$ e_{\gamma_{n+1}}^{*}(x) > \frac{\theta-\varepsilon}{N} \cdot n + \frac{\theta}{N} \frac{1}{q+1} \cdot (q+1) \Big( 1 - \frac{\varepsilon}{2\theta} \Big)^{2} > \frac{\theta-\varepsilon}{N} \cdot (n+1). $$
Therefore, the choice of $ \gamma_{n+1} $ and $ r(n+1) $ works indeed.

Finally, we get $ \Vert \sum_{j=1}^{r(N)} x_{j} \Vert \geq e^{*}_{\gamma_{N}}(\sum_{j=1}^{r(N)} x_{j}) > \frac{\theta-\varepsilon}{N} \cdot N = \theta - \varepsilon $.
\end{proof}

\begin{proposition} \label{propXMisomc0}
{\rm (1)} If $ \mathcal{M} $ contains an infinite set, then $ \mathfrak{X}_{\mathcal{M}} $ is isomorphic to $ c_{0} $.

{\rm (2)} If $ \mathcal{M} $ consists of finite sets only, then $ \mathfrak{X}_{\mathcal{M}} $ does not contain an isomorphic copy of $ c_{0} $ (cf. \cite[Proposition~5.7(iii)]{arggasmot}).
\end{proposition}

\begin{proof}
(1) Let $ c = (N-2\theta)/2N, C = 2N(N+\theta)/(N-2\theta)(1-\theta \omega) $. If $ \{ n_{1}, n_{2}, \dots \} \in \mathcal{M} $ for an infinite increasing sequence $ n_{1} < n_{2} < \dots $, then it follows from Claim~\ref{XMclaim1} and Claim~\ref{XMclaim3} that
$$ c \cdot \sup_{k\in\mathbb{N} \cup \{ 0 \}} \Vert P_{E_{k}}u \Vert \leq \Vert u \Vert \leq C \cdot \sup_{k\in\mathbb{N} \cup \{ 0 \}} \Vert P_{E_{k}}u \Vert, \quad u \in \mathfrak{X}_{\mathcal{M}}, $$
where $ E_{0} = \{ 0, \dots, n_{1} - 1 \} $ and $ E_{k} = \{ n_{k}, \dots , n_{k+1} - 1 \} $. Hence the space $ \mathfrak{X}_{\mathcal{M}} $ is isomorphic to the $ c_{0} $-sum of a sequence of finite-dimensional spaces.

It follows that $ \mathfrak{X}_{\mathcal{M}} $ is isomorphic to a subspace of $ c_{0} $. By Theorem~\ref{thmros}, since it is a $ \mathcal{L}_{\infty} $-space at the same time, it is isomorphic to $ c_{0} $. (Let us remark that a simpler argument is provided in the proof of \cite[Corollary~1]{johzip}).

(2) Let us fix $ \varepsilon > 0 $ such that $ \theta - \varepsilon > 1 $. By a simple induction argument, we show that Claim~\ref{XMclaim4} holds in fact with the conclusion $ \Vert \sum_{j = 1}^{r} x_{j} \Vert > (\theta - \varepsilon)^{n} $. Assume that this statement holds for $ n $. If $ x_{1}, x_{2}, \dots $ is such a sequence as in Claim~\ref{XMclaim4}, then we can apply the induction hypothesis again and again to find numbers $ 0 = s_{1} < s_{2} < \dots $ such that $ \Vert \sum_{j = s_{k}+1}^{s_{k+1}} x_{j} \Vert > (\theta - \varepsilon)^{n} $ for every $ k \in \mathbb{N} $. Applying (unmodified) Claim~\ref{XMclaim4} on the sequence $ \frac{1}{(\theta - \varepsilon)^{n}} \sum_{j = s_{k}+1}^{s_{k+1}} x_{j} $, we obtain $ r \in \mathbb{N} $ such that $ \frac{1}{(\theta - \varepsilon)^{n}} \Vert \sum_{k = 1}^{r} \sum_{j = s_{k}+1}^{s_{k+1}} x_{j} \Vert > \theta - \varepsilon $. That is, $ \Vert \sum_{j = 1}^{s_{r+1}} x_{j} \Vert > (\theta - \varepsilon)^{n+1} $, which proves the statement for $ n+1 $.

Suppose that $ \mathfrak{X}_{\mathcal{M}} $ contains an isomorphic copy of $ c_{0} $. Then there is a sequence $ u_{1}, u_{2}, \dots $ in $ \mathfrak{X}_{\mathcal{M}} $ such that $ \Vert u_{i} \Vert \geq 1 $ and
$$ \mathrm{range} \, u_{1} < \mathrm{range} \, u_{2} < \dots $$
that is equivalent to the standard basis of $ c_{0} $. The sequence $ x_{j} = u_{2j} $ satisfies the assumption of Claim~\ref{XMclaim4}. Therefore, for every $ n \in \mathbb{N} $, there is $ r \in \mathbb{N} $ such that $ \Vert \sum_{j = 1}^{r} u_{2j} \Vert = \Vert \sum_{j = 1}^{r} x_{j} \Vert > (\theta - \varepsilon)^{n} $. As $ (\theta - \varepsilon)^{n} $ can be arbitrarily large, the sequence $ u_{1}, u_{2}, \dots $ can not be equivalent to the standard basis of $ c_{0} $.
\end{proof}

We introduce here one more result that will be useful later for proving Theorem~\ref{thmmain2}. First, we show that $ \mathfrak{X}_{\mathcal{M}} $ has a property called the \emph{boundedly complete skipped blocking property} in \cite{bourgain}.

\begin{claim} \label{XMclaim5}
Suppose that $ \mathcal{M} $ consists of finite sets only. Let $ a_{1}, a_{2}, \dots $ be a sequence in $ \mathfrak{X}_{\mathcal{M}} $ such that
$$ \mathrm{range} \, a_{1} < m_{1} < \mathrm{range} \, a_{2} < m_{2} < \mathrm{range} \, a_{3} < m_{3} < \dots $$
for some $ m_{1}, m_{2}, \dots $. If the sequence of partial sums $ \sum_{i=1}^{s} a_{i}, s = 1, 2, \dots, $ is bounded, then the sum $ \sum_{i=1}^{\infty} a_{i} $ is convergent.
\end{claim}

\begin{proof}
Assume that the sum $ \sum_{i=1}^{\infty} a_{i} $ is not convergent. Then there is $ \delta > 0 $ such that, for every $ s \in \mathbb{N} \cup \{ 0 \} $, there is $ r \in \mathbb{N} $ such that $ \Vert \sum_{i=s+1}^{r} a_{i} \Vert > \delta $. This allows us to find numbers $ 0 = s_{1} < s_{2} < \dots $ such that $ \Vert \sum_{i=s_{k}+1}^{s_{k+1}} a_{i} \Vert > \delta $ for every $ k \in \mathbb{N} $. The sequence
$$ x_{k} = \frac{1}{\delta} \sum_{i=s_{k}+1}^{s_{k+1}} a_{i}, \quad k = 1, 2, \dots, $$
satisfies the assumption of Claim~\ref{XMclaim4}. By the argument from the proof of Proposition~\ref{propXMisomc0}(2), for every $ n \in \mathbb{N} $, there is $ r \in \mathbb{N} $ such that $ \Vert \sum_{k = 1}^{r} x_{k} \Vert > (\theta - \varepsilon)^{n} $. That is, $ \Vert \sum_{i = 1}^{s_{r+1}} a_{i} \Vert > \delta \cdot (\theta - \varepsilon)^{n} $. For this reason, the sequence of partial sums $ \sum_{i=1}^{s} a_{i}, s = 1, 2, \dots, $ is not bounded.
\end{proof}

\begin{proposition}[{cf. \cite[Remark~5.9]{arggasmot}}] \label{propXMreflsatu}
If $ \mathcal{M} $ consists of finite sets only and the dual space $ \mathfrak{X}_{\mathcal{M}}^{*} $ is separable, then every infinite-dimensional subspace of $ \mathfrak{X}_{\mathcal{M}} $ contains an infinite-dimensional reflexive subspace.
\end{proposition}

\begin{proof}
This follows from Claim~\ref{XMclaim5} and \cite[Corollary~2.6(1)]{bourgain}.
\end{proof}

\section{Conclusion}

As well as in the previous section, we fix $ N \geq 3, 1 < \theta < N/2 $ and $ 0 < \omega < 1/\theta $. For every $ \mathcal{M} \in K(\mathcal{P}(\mathbb{N})) $, we consider the space $ \mathfrak{X}_{\mathcal{M}} $ constructed for $ N, \theta, \omega $ and $ \mathcal{M} $. What remains to show is that the mapping $ \mathcal{M} \mapsto \mathfrak{X}_{\mathcal{M}} $ is measurable.

\begin{lemma} \label{lemmselectXM}
There exists a Borel mapping $ \mathfrak{S} : K(\mathcal{P}(\mathbb{N})) \to \mathcal{SE}(C([0, 1])) $ such that $ \mathfrak{S}(\mathcal{M}) $ is isometric to $ \mathfrak{X}_{\mathcal{M}} $ for every $ \mathcal{M} \in K(\mathcal{P}(\mathbb{N})) $.
\end{lemma}

\begin{proof}
For $ \mathcal{M} \in K(\mathcal{P}(\mathbb{N})) $, let us denote the objects from Section~\ref{sec:bourdelb} by $ \Delta_{p}^{\mathcal{M}}, \Gamma_{p}^{\mathcal{M}}, c_{\gamma}^{*\mathcal{M}} $, etc. For every $ \mathcal{M} \in K(\mathcal{P}(\mathbb{N})) $, it holds that $ \Delta_{0}^{\mathcal{M}} = \Gamma_{0}^{\mathcal{M}} = \{ 0 \} $ and $ i_{0, 0}^{\mathcal{M}} $ and $ P_{\{ 0 \}}^{0, \mathcal{M}} $ are the identity on $ \ell_{\infty}(\{ 0 \}) $. We can easily prove by induction on $ q \in \mathbb{N} \cup \{ 0 \} $ the following observation:

\emph{If $ \mathcal{M}_{1}, \mathcal{M}_{2} \in K(\mathcal{P}(\mathbb{N})) $ are two compact systems such that
$$ \big\{ A \cap \{ 0, \dots, q \} : A \in \mathcal{M}_{1} \big\} = \big\{ A \cap \{ 0, \dots, q \} : A \in \mathcal{M}_{2} \big\}, $$
then we have $ \Delta_{q+1}^{\mathcal{M}_{1}} = \Delta_{q+1}^{\mathcal{M}_{2}}, \Gamma_{q+1}^{\mathcal{M}_{1}} = \Gamma_{q+1}^{\mathcal{M}_{2}} $, $ \mathrm{age}^{\mathcal{M}_{1}}(\gamma) = \mathrm{age}^{\mathcal{M}_{2}}(\gamma) $ and $ c_{\gamma}^{*\mathcal{M}_{1}} = c_{\gamma}^{*\mathcal{M}_{2}} $ for $ \gamma \in \Delta_{q+1}^{\mathcal{M}_{1}} $, $ i_{p,q+1}^{\mathcal{M}_{1}} = i_{p,q+1}^{\mathcal{M}_{2}} $ for $ 0 \leq p \leq q+1 $ and $ P_{E}^{q+1, \mathcal{M}_{1}} = P_{E}^{q+1, \mathcal{M}_{2}} $ for an interval $ E $ of $ \{ 0, 1, \dots, q+1 \} $.}

Let $ \gamma_{1}^{\mathcal{M}} = 0 $ (i.e., $ \Gamma_{0}^{\mathcal{M}} = \{ \gamma_{1}^{\mathcal{M}} \} $) for all $ \mathcal{M} \in K(\mathcal{P}(\mathbb{N})) $. By a recursive procedure, for every $ q \in \mathbb{N} \cup \{ 0 \} $, we enumerate sets $ \Gamma_{q+1}^{\mathcal{M}} $ in the way that we choose the same enumeration
$$ \Delta_{q+1}^{\mathcal{M}} = \big\{ \gamma_{|\Gamma_{q}^{\mathcal{M}}|+1}^{\mathcal{M}}, \gamma_{|\Gamma_{q}^{\mathcal{M}}|+2}^{\mathcal{M}}, \dots, \gamma_{|\Gamma_{q+1}^{\mathcal{M}}|}^{\mathcal{M}} \big\} $$
for $ \mathcal{M} $'s which have the same set $ \{ A \cap \{ 0, \dots, q \} : A \in \mathcal{M} \} $. We claim that Lemma~\ref{lemmselect} can be applied on $ \Xi = K(\mathcal{P}(\mathbb{N})) $,
$$ X_{\mathcal{M}} = \mathfrak{X}_{\mathcal{M}} \quad \textrm{and} \quad x^{\mathcal{M}}_{k} = d_{\gamma_{k}^{\mathcal{M}}}^{\mathcal{M}}, \; k \in \mathbb{N}. $$
Let $ n \in \mathbb{N} $ and $ \lambda_{1}, \dots, \lambda_{n} \in \mathbb{R} $ be given. Clearly, $ \gamma_{k}^{\mathcal{M}} \in \Gamma_{n}^{\mathcal{M}} $ for $ 1 \leq k \leq n $, as $ \Gamma_{n}^{\mathcal{M}} $ has at least $ n $ elements. We have
$$ \sum_{k=1}^{n} \lambda_{k} x^{\mathcal{M}}_{k} = P_{[0, n]}^{\mathcal{M}} \Big( \sum_{k=1}^{n} \lambda_{k} x^{\mathcal{M}}_{k} \Big) = (i_{n}^{\mathcal{M}} \circ r_{n}^{\mathcal{M}}) \Big( \sum_{k=1}^{n} \lambda_{k} x^{\mathcal{M}}_{k} \Big), $$
and so
$$ \Big\Vert \sum_{k=1}^{n} \lambda_{k} x^{\mathcal{M}}_{k} \Big\Vert = \lim_{q \to \infty} \Big\Vert (i_{n, q+1}^{\mathcal{M}} \circ r_{n}^{\mathcal{M}}) \Big( \sum_{k=1}^{n} \lambda_{k} x^{\mathcal{M}}_{k} \Big) \Big\Vert. $$
It remains to realize that $ \mathcal{M} \mapsto \Vert (i_{n, q+1}^{\mathcal{M}} \circ r_{n}^{\mathcal{M}}) ( \sum_{k=1}^{n} \lambda_{k} x^{\mathcal{M}}_{k} ) \Vert $ is a Borel function for every $ q \geq n-1 $. This function is continuous in fact. Indeed, due to the above observation and the method of our enumeration, the norm $ \Vert (i_{n, q+1}^{\mathcal{M}} \circ r_{n}^{\mathcal{M}}) ( \sum_{k=1}^{n} \lambda_{k} x^{\mathcal{M}}_{k} ) \Vert $ depends only on $ \{ A \cap \{ 0, \dots, q \} : A \in \mathcal{M} \} $. For this reason, $ K(\mathcal{P}(\mathbb{N})) $ can be decomposed into finitely many clopen sets on which this norm is constant.
\end{proof}

We are going to prove the complexity results proclaimed in the introduction. The following classical result (see e.g. \cite[(27.4)]{kechris}) is behind all of them.

\begin{theorem}[Hurewicz] \label{thmhur}
The set
$$ \mathfrak{H} = \Big\{ \mathcal{M} \in K(\mathcal{P}(\mathbb{N})) : \textnormal{$ \mathcal{M} $ contains an infinite set} \Big\} $$
is a complete analytic subset of $ K(\mathcal{P}(\mathbb{N})) $. In particular, it is not Borel.
\end{theorem}

\begin{proof}[Proof of Theorem \ref{thmmain}]
It is easy to show that the class of all Banach spaces $ X $ isomorphic to $ c_{0} $ (shortly $ X \simeq c_{0} $) is analytic (see \cite[Theorem~2.3]{bossard}). Let us show that it is hard analytic. Let $ \mathfrak{S} $ be a mapping provided by Lemma~\ref{lemmselectXM}. By Proposition~\ref{propXMisomc0}, we have
$$ \mathfrak{S}(\mathcal{M}) \simeq c_{0} \quad \Leftrightarrow \quad \textnormal{$ \mathcal{M} $ contains an infinite set}. $$
Now, it is sufficient to use Theorem~\ref{thmhur} and Lemma~\ref{lemmhardset}.
\end{proof}

\begin{proof}[Proof of Theorem \ref{thmmain2}]
We check first that
$$ \mathfrak{X}_{\mathcal{M}} \oplus X \simeq c_{0} \oplus X \quad \Leftrightarrow \quad \textnormal{$ \mathcal{M} $ contains an infinite set,} $$
provided that $ \mathfrak{X}_{\mathcal{M}}^{*} $ is separable (we need to apply Proposition~\ref{propXMreflsatu} in the case that $ X $ satisfies the second condition).

The implication \uv{$ \Leftarrow $} follows simply from Proposition~\ref{propXMisomc0}(1). Concerning the implication \uv{$ \Rightarrow $}, we need to consider the conditions for $ X $ separately. Let us assume that $ \mathcal{M} $ consists of finite sets only, and so that $ \mathfrak{X}_{\mathcal{M}} $ does not contain a subspace isomorphic to $ c_{0} $ by Proposition~\ref{propXMisomc0}(2). If $ X $ does not contain a subspace isomorphic to $ c_{0} $, then $ \mathfrak{X}_{\mathcal{M}} \oplus X $ has the same property by Fact~\ref{factrams}(1), and thus $ \mathfrak{X}_{\mathcal{M}} \oplus X $ is not isomorphic to $ c_{0} \oplus X $. If $ X $ does not contain an infinite-dimensional reflexive subspace, then $ c_{0} \oplus X $ has the same property by Fact~\ref{factrams}(2), and thus $ c_{0} \oplus X $ is not isomorphic to $ \mathfrak{X}_{\mathcal{M}} \oplus X $, as $ \mathfrak{X}_{\mathcal{M}} $ has an infinite-dimensional subspace by Proposition~\ref{propXMreflsatu}. If $ X $ is a subspace of a space with an unconditional basis, then $ c_{0} \oplus X $ has the same property, and thus $ c_{0} \oplus X $ is not isomorphic to $ \mathfrak{X}_{\mathcal{M}} \oplus X $, as $ \mathfrak{X}_{\mathcal{M}} $ is not isomorphic to a subspace of a space with an unconditional basis (otherwise it would be isomorphic to $ c_{0} $ by Theorem~\ref{thmros}).

We realize that there is a Borel subset $ \mathfrak{B} $ of $ K(\mathcal{P}(\mathbb{N})) $ containing the analytic non-Borel set $ \mathfrak{H} $ from Theorem~\ref{thmhur} such that $ \mathfrak{X}_{\mathcal{M}}^{*} $ is separable for every $ \mathcal{M} \in \mathfrak{B} $. It is sufficient to apply the Lusin separation theorem (see e.g. \cite[(14.7)]{kechris}), as $ \mathfrak{H} $ is disjoint from $ \{ \mathcal{M} \in K(\mathcal{P}(\mathbb{N})) : \mathfrak{X}_{\mathcal{M}}^{*} \textrm{ is not separable} \} $ that is analytic by Lemma~\ref{lemmselectXM} and \cite[Corollary~3.3]{bossard}. (In fact, concrete Borel subsets are available, for instance $ \mathfrak{B} = \{ \mathcal{M} \in K(\mathcal{P}(\mathbb{N})) : \mathrm{Sz}(\mathfrak{X}_{\mathcal{M}}) \leq \omega_{0} \} $).

It is easy to find a Borel mapping $ \mathfrak{T} : K(\mathcal{P}(\mathbb{N})) \to \mathcal{SE}(C([0, 1])) $ such that $ \mathfrak{T}(\mathcal{M}) $ is isomorphic to $ \mathfrak{X}_{\mathcal{M}} \oplus X $ for every $ \mathcal{M} \in K(\mathcal{P}(\mathbb{N})) $ (if $ I : C([0, 1]) \oplus X \to C([0, 1]) $ is an isomorphism, we can put $ \mathfrak{T}(\mathcal{M}) = I(\mathfrak{S} (\mathcal{M}) \oplus X) $, where $ \mathfrak{S} $ is a mapping given by Lemma~\ref{lemmselectXM}). We have
$$ \mathfrak{T}(\mathcal{M}) \simeq c_{0} \oplus X \quad \Leftrightarrow \quad \textnormal{$ \mathcal{M} $ contains an infinite set} $$
for every $ \mathcal{M} \in \mathfrak{B} $. Hence, it follows finally from Theorem~\ref{thmhur} that the isomorphism class of $ c_{0} \oplus X $ is not Borel.

If $ X $ does not contain an infinite-dimensional reflexive subspace or $ X $ is a subspace of a space with an unconditional basis, then the same proof works for the class of all Banach spaces $ Y $ that are isomorphic to a subspace of $ c_{0} \oplus X $ (shortly $ Y \hookrightarrow c_{0} \oplus X $). Indeed, it follows easily from the above arguments that
$$ \mathfrak{X}_{\mathcal{M}} \oplus X \hookrightarrow c_{0} \oplus X \quad \Leftrightarrow \quad \textnormal{$ \mathcal{M} $ contains an infinite set,} $$
provided that $ \mathfrak{X}_{\mathcal{M}}^{*} $ is separable.

Thus, there is the only remaining case that $ X $ does not contain an isomorphic copy of $ c_{0} $. Let us assume that the class of all Banach spaces isomorphic to a subspace of $ c_{0} \oplus X $ is Borel. Then, by Lemma~\ref{lemmselectXM},
$$ \mathfrak{C} = \Big\{ \mathcal{M} \in K(\mathcal{P}(\mathbb{N})) : \mathfrak{X}_{\mathcal{M}} \hookrightarrow c_{0} \oplus X \Big\} $$
is a Borel subset of $ K(\mathcal{P}(\mathbb{N})) $. We have $ \mathfrak{H} \subset \mathfrak{C} $ by Proposition~\ref{propXMisomc0}(1). Let us further consider
$$ \mathfrak{D} = \Big\{ \mathcal{M} \in \mathfrak{C} : \exists Z, \mathrm{dim} \, Z = \infty, Z \hookrightarrow \mathfrak{X}_{\mathcal{M}}, Z \hookrightarrow X \Big\}. $$
Using Lemma~\ref{lemmselectXM} and \cite[Theorem~2.3(ii)]{bossard}, it is straightforward to check that $ \mathfrak{D} $ is analytic. Let us show that
$$ \mathfrak{C} \setminus \mathfrak{H} \subset \mathfrak{D}. $$
Let $ \mathcal{M} \in \mathfrak{C} \setminus \mathfrak{H} $. By Lemma~\ref{lemmrams}, $ \mathfrak{X}_{\mathcal{M}} $ has an infinite-dimensional subspace $ Z $ such that $ Z \hookrightarrow c_{0} $ or $ Z \hookrightarrow X $. Since every infinite-dimensional subspace of $ c_{0} $ contains an isomorphic copy of $ c_{0} $, the possibility $ Z \hookrightarrow c_{0} $ (which would imply $ c_{0} \hookrightarrow Z \subset \mathfrak{X}_{\mathcal{M}} $) is excluded by Proposition~\ref{propXMisomc0}(2). Thus, $ Z $ witnesses that $ \mathcal{M} \in \mathfrak{D} $.

By Theorem~\ref{thmhur}, the set $ \mathfrak{C} \setminus \mathfrak{H} $ is not analytic (by the Lusin separation theorem, $ \mathfrak{H} $ would be Borel in the opposite case). For this reason, the inclusion $ \mathfrak{C} \setminus \mathfrak{H} \subset \mathfrak{D} $ is proper. Pick some $ \mathcal{M} \in \mathfrak{H} \cap \mathfrak{D} $. There is an infinite-dimensional Banach space $ Z $ such that $ Z \hookrightarrow \mathfrak{X}_{\mathcal{M}} $ and $ Z \hookrightarrow X $. Since $ \mathfrak{X}_{\mathcal{M}} \simeq c_{0} $ and every infinite-dimensional subspace of $ c_{0} $ contains an isomorphic copy of $ c_{0} $, we have $ c_{0} \hookrightarrow Z \hookrightarrow X $. This contradicts the assumption on $ X $.

(Let us remark that the last part of the proof does not require the knowledge of the spaces $ \mathfrak{X}_{\mathcal{M}} $, as much simpler spaces $ T^{*}[\mathcal{M}, \frac{1}{2}] $ studied e.g. in \cite[Section~3]{kurka} can be used in the definition of $ \mathfrak{C} $ as well).
\end{proof}

\begin{proof}[Proof of Theorem \ref{thmmain3}]
It was shown in \cite[5.7]{braga}, \cite[p.~367]{godefroyprobl} and \cite[p.~368]{godefroyprobl} that the classes are analytic. Let us show that they are not Borel. We have already seen in the proof of Theorem~\ref{thmmain} that it is sufficient to check that
\begin{align*}
\textnormal{$ \mathcal{M} $ contains an infinite set} & \Leftrightarrow \textrm{$ \mathfrak{X}_{\mathcal{M}} $ has an unconditional basis} \\
 & \Leftrightarrow \textrm{$ \mathfrak{X}_{\mathcal{M}} \hookrightarrow Z $, where $ Z $ has an unc. basis} \\
 & \Leftrightarrow \textrm{$ \mathfrak{X}_{\mathcal{M}} $ is isomorphic to a $ C(K) $ space.}
\end{align*}
As $ c_{0} $ satisfies all three properties, implications \uv{$ \Rightarrow $} follow simply from Proposition~\ref{propXMisomc0}(1). Considering Proposition~\ref{propXMisomc0}(2), to prove implications \uv{$ \Leftarrow $}, we need to verify that any of those three conditions implies that $ \mathfrak{X}_{\mathcal{M}} $ contains an isomorphic copy of $ c_{0} $. It is sufficient to use Theorem~\ref{thmros} and the fact that every infinite-dimensional separable $ C(K) $ space contains an isomorphic copy of $ c_{0} $.
\end{proof}


\begin{thebibliography}{99}
\bibitem{argdel} S. A. Argyros and I. Deliyanni, \emph{Examples of asymptotic $ \ell_{1} $ Banach spaces}, Trans. Amer. Math. Soc. {\bf 349}, no. 3 (1997), 973--995.
\bibitem{arggasmot} S. A. Argyros, I. Gasparis and P. Motakis, \emph{On the structure of separable $ \mathcal{L}_{\infty} $-spaces}, Mathematika {\bf 62}, no. 3 (2016), 685--700.
\bibitem{arghay} S. A. Argyros and R. G. Haydon, \emph{A hereditarily indecomposable $ \mathcal{L}_{\infty} $-space that solves the scalar-plus-compact problem}, Acta Math. {\bf 206}, no. 1 (2011), 1--54.
\bibitem{braga} B. M. Braga, \emph{On the complexity of some inevitable classes of separable Banach spaces}, J. Math. Anal. Appl. {\bf 431}, no. 1 (2015), 682--701.
\bibitem{bossard} B. Bossard, \emph{A coding of separable Banach spaces. Analytic and coanalytic families of Banach spaces}, Fund. Math. {\bf 172}, no. 2 (2002), 117--152.
\bibitem{bourgain} J. Bourgain, \emph{New classes of $ \mathcal{L}^{p} $-spaces}, Lecture notes in mathematics {\bf 889}, Springer, 1981.
\bibitem{bourdelb} J. Bourgain and F. Delbaen, \emph{A class of special $ \mathcal{L}_{\infty} $ spaces}, Acta Math. {\bf 145}, no. 1 (1980), 155--176.
\bibitem{dodos} P. Dodos, \emph{Banach spaces and descriptive set theory: selected topics}, Lecture notes in mathematics {\bf 1993}, Springer, 2010.
\bibitem{fhhmpz} M.~Fabian, P.~Habala, P.~H\'ajek, V.~Montesinos Santaluc\'ia, J.~Pelant and V.~Zizler, \emph{Functional analysis and infinite-dimensional geometry}, CMS Books in~Mathematics {\bf 8}, Springer, 2001.
\bibitem{figjoh} T. Figiel and W. B. Johnson, \emph{A uniformly convex Banach space which contains no $ \ell_{p} $}, Compos. Math. {\bf 29}, no. 2 (1974), 179--190.
\bibitem{ghawadrah} G. Ghawadrah, \emph{Non-isomorphic complemented subspaces of the reflexive Orlicz function spaces $ L^{\Phi}[0, 1] $}, Proc. Amer. Math. Soc. {\bf 144}, no. 1 (2016), 285--299.
\bibitem{ghawadrah2} G. Ghawadrah, \emph{The descriptive complexity of the family of Banach spaces with the bounded approximation property}, Houston J. Math. {\bf 43}, no. 2 (2017), 395--401.
\bibitem{godefroyprobl} G. Godefroy, \emph{Analytic sets of Banach spaces}, Rev. R. Acad. Cien. Serie A. Mat. {\bf 104}, no. 2 (2010), 365--374.
\bibitem{godefroycompl} G. Godefroy, \emph{The complexity of the isomorphism class of some Banach spaces}, J. Nonlinear Convex Anal. {\bf 18}, no. 2 (2017), 231--240.
\bibitem{godefroycompl2} G. Godefroy, \emph{The isomorphism classes of $ \ell_{p} $ are Borel}, Houston J. Math. {\bf 43}, no. 3 (2017), 947--951.
\bibitem{gokala} G. Godefroy, N. J. Kalton and G. Lancien, \emph{Szlenk indices and uniform homeomorphisms}, Trans. Amer. Math. Soc. {\bf 353}, no. 10 (2001), 3895--3918.
\bibitem{gosr} G. Godefroy and J. Saint-Raymond, \emph{Descriptive complexity of some isomorphic classes of Banach spaces}, to appear.
\bibitem{johzip} W. B. Johnson and M. Zippin, \emph{On subspaces of quotients of $ (\sum G_{n})_{\ell_{p}} $ and $ (\sum G_{n})_{c_{0}} $}, Israel J. Math. {\bf 13}, no. 3 (1972), 311--316.
\bibitem{kechris} A. S. Kechris, \emph{Classical descriptive set theory}, Graduate Texts in Mathematics {\bf 156}, Springer-Verlag, 1995.
\bibitem{kurka} O. Kurka, \emph{Tsirelson-like spaces and complexity of classes of Banach spaces}, Rev. R. Acad. Cien. Serie A. Mat. (to appear).
\bibitem{rosenthal} H. P. Rosenthal, \emph{A characterization of $ c_{0} $ and some remarks concerning the Grothendieck property}, Texas functional analysis seminar 1982--1983 (Austin, Tex.), 95--108, Longhorn Notes, Univ. Texas Press, 1983. 
\bibitem{tsirelson} B. S. Tsirelson, \emph{Not every Banach space contains an imbedding of $ \ell_{p} $ or $ c_{0} $}, Funct. Anal. Appl. {\bf 8}, no. 2 (1974), 138--141.
\end{thebibliography}
\end{document}